\documentclass[12pt,reqno]{amsart}
\usepackage[top=1in, bottom=1in, left=1in, right=1in]{geometry}
\usepackage{fourier, amsthm, amssymb, amsmath, amsfonts, amsthm, bm, graphicx, mathrsfs}
\usepackage{latexsym,wasysym}
\usepackage{amscd}
\usepackage{float}
\usepackage[all]{xy}
\usepackage[usenames,dvipsnames]{color}
\usepackage{tikz}
\usepackage[pagebackref=true,pdftex]{hyperref}

 \numberwithin{equation}{section}

\newcommand{\excise}[1]{}

\newtheorem{theorem}{Theorem}[section]
\newtheorem{proposition}[theorem]{Proposition}
\newtheorem{lemma}[theorem]{Lemma}
\newtheorem{corollary}[theorem]{Corollary}
\newtheorem{definition}[theorem]{Definition}

\theoremstyle{definition}

\newcommand\<{\langle}
\newcommand\CC{{\mathbb C}}
\newcommand\EE{{\mathbb E}}
\newcommand\FF{{\mathbb F}}

\newcommand\NN{{\mathbb N}}
\newcommand\QQ{{\mathbb Q}}
\newcommand\RR{{\mathbb R}}

\newcommand\ZZ{{\mathbb Z}}

\newcommand\cJ{{\mathcal J}}

\newcommand\del{\partial}

\newcommand\codim{{\rm codim}}

\newcommand\conv{{\rm conv}}

\newcommand\ini{{\rm in}}

\newcommand\rank{{\rm rank\,}}

\newcommand\Sol{{\rm Sol}}

\newcommand\vol{{\rm vol}}
\newcommand\Var{{\rm Var}}

\newcommand\inww{{\rm in}_{(-w,w)}}
\newcommand\inw{{\rm in}_w}
\newcommand\HA{{H_A(\beta)}}

\newcommand\minus{\smallsetminus}

\renewcommand\>{\rangle}

\newcommand*{\defeq}{\mathrel{\vcenter{\baselineskip0.5ex \lineskiplimit0pt
                     \hbox{\scriptsize.}\hbox{\scriptsize.}}}%
                     =}

\begin{document}
\mbox{}
\title{On the parametric behavior of $A$-hypergeometric series}

\author{Christine Berkesch Zamaere}
\address{School of Mathematics, University of Minnesota, Minneapolis, MN 55455.}
\email{cberkesc@math.umn.edu}

\author{Jens~Forsg{\aa}rd}
\address{Department of Mathematics \\
Texas A\&M University \\ College Station, TX 77843.}
\email{jensf@math.tamu.edu}

\author{Laura Felicia Matusevich}
\address{Department of Mathematics \\
Texas A\&M University \\ College Station, TX 77843.}
\email{laura@math.tamu.edu}

\thanks{
CBZ was partially supported by NSF Grant DMS 1440537. 
LFM was partially supported by NSF grants DMS 1001763 and DMS 1500832. 
JF was partially supported by G. S. Magnusson Fund of the Royal Swedish Academy of Sciences
}
\begin{abstract}
We describe the parametric behavior of the series solutions of an
$A$-hyper\-geo\-metric system. More precisely, we construct
explicit
stratifications of the parameter space such that, on each stratum, the
series solutions of the system are
holomorphic.  
\end{abstract}
\maketitle

\setcounter{section}{0}
\section{Introduction}
\label{sec:intro}

A classical object of study, the Gauss hypergeometric equation is
\begin{equation}
\label{eqn:Gauss}
z (1-z) \frac{d^2F}{dz^2} + (c-z(a+b+1))\frac{dF}{dz} -abF = 0 \,;
\qquad a, b, c \in \CC.
\end{equation}
The quantities $a,b,c$ are considered parameters, and $z$ is
considered a variable. For $c\notin \ZZ$, the functions
\begin{equation}
\label{eqn:GaussSolutions}
F(z;a,b,c) = \sum_{n=0}^\infty
\left[ \prod_{\ell=0}^{n-1}\frac{(a+\ell)(b+\ell)}{(c+\ell)} \right]\cdot
\frac{z^n}{n!} \; ; \quad z^{1-c}F(z;a+1-c,b+1-c,2-c)
\end{equation}
form a basis of the solution space of~\eqref{eqn:Gauss} in a
neighborhood of the origin $0\in \CC$. It is not hard to see that as
functions of $a,b,c$, the series~\eqref{eqn:GaussSolutions} are
meromorphic; the first has poles for $c \in \ZZ_{\leq 0}$, and the
second one for $c \in \ZZ_{>1}$. In order to solve~\eqref{eqn:Gauss}
for $c\in \ZZ$, one takes derivatives of the
series~\eqref{eqn:GaussSolutions} with respect to the parameters. For
instance, when $c=1$, both series~\eqref{eqn:GaussSolutions}
coincide, and 
\[
F(z;a,b,1) \log(z) + \sum_{n=0}^\infty \left[ \frac{\del}{\del h}
  \prod_{\ell=0}^{n-1}\frac{(a+h+\ell)(b+h+\ell)}{(c+h+\ell)(1+h+\ell)}
\right]_{h=0} z^n 
\]
is the second function needed to span the solution
of~\eqref{eqn:Gauss} near $0\in \CC$. Since it is
obtained as a derivative, the series above is an entire function of
$a,b$, and holomorphic for $z$ in any simply connected subset of the
punctured disc $0<|z|<1$ in $\CC$. In this way, it is possible to
understand the solutions of~\eqref{eqn:Gauss} as functions of $(a,b,c)
\in \CC^3$.

The goal of this article is to perform a similar analysis for
certain generalized hypergeometric systems in several variables:
the \emph{$A$-hypergeometric systems} of Gelfand, Graev,
Kapranov and Zelevinsky
(Definition~\ref{def:GKZ}).

We adopt the convention that $\NN = \{0,1,2,3,\dots\}$. Throughout
this article, 
let $A$ be a $d\times n$ integer matrix of rank $d$ whose columns, denoted by $a_1,\dots,a_n$, span
$\ZZ^d$ as a lattice. We write $A = [a_{ij}]$ where $i = 1, \dots, d$ and $j = 1, \dots, n$.
In Sections~\ref{sec:SST},~\ref{sec:LogFreeHolomorphic},~\ref{sec:LogSeriesParametric}
and~\ref{sec:LogSeriesHolomorphic}, we impose the additional assumption
that $A$ is \emph{homogeneous}, i.e., that $(1,\dots,1)\in \QQ^n$ is in the
$\QQ$-row span of $A$. We return to the inhomogeneous case in Section~\ref{sec:confluentSolutions}.

The \emph{Weyl algebra} $D$ is the ring of differential operators on
$\CC^n$. In other words, $D$ is the quotient of the free associative $\CC$-algebra
generated by the variables $x_1,\dots,x_n$, and $\del_1,\dots,\del_n$,
by the two-sided ideal   
\[
\left\< x_k x_\ell = x_\ell x_k,\ 
  \del_k \del_\ell = \del_\ell \del_k, \ 
  \del_k x_\ell = x_\ell \del_k +\delta_{k\ell}
  \mid k,\ell\in\{1,\dots,n\} \right\>,
\]
where $\delta_{k\ell}$ denotes the Kronecker $\delta$-function. 

The \emph{toric ideal} of the matrix $A$ is defined by
\[
I_A = \left\< \del^u - \del^v 
  \mid u,v\in\NN^n, \, 
  Au=Av \right\> \subseteq \CC[\del]. 
\]
For each parameter $\beta\in\CC^d$, the \emph{Euler operators}
associated to $A$ and $\beta$ are
$E-\beta = \{E_i-\beta_i\}_{i=1}^d$, where
\[
E_i = \sum_{j=1}^na_{ij}x_j\del_j ,\qquad i = 1,\dots, d.
\] 
\begin{definition}
\label{def:GKZ}
The 
\emph{$A$-hypergeometric system} with parameter $\beta$ is
given by
\begin{align*}
\label{eqn:HAbeta}
\HA \defeq D\cdot (I_A + \< E-\beta\>).
\end{align*}
The $D$-module $D/H_A(\beta)$ is known as the \emph{$A$-hypergeometric
  $D$-module with parameter $\beta$}.
\end{definition}

$A$-hypergeometric systems~\cite{GGZ,gkz88,GKZ} were introduced by Gelfand, Graev, Kapranov
and Zelevinsky.
We study here \emph{$A$-hy\-per\-geo\-metric functions}, the solutions of
$A$-hy\-per\-geo\-me\-tric systems.
Let $x$ be a nonsingular point of the system $H_A(\beta)$.
A \emph{solution of $D/\HA$ at $x$} is a germ $\varphi$ of a holomorphic function
such that $P\bullet \varphi = 0$ for each $P\in H_A(\beta)$.
The solutions of $D/\HA$ at $x$ form a $\CC$-vector space, which we denote by $\Sol_x(\HA)$. 
The \emph{rank} of $D/\HA$, denoted $\rank D/\HA$, is the vector space
dimension of $\Sol_x(\HA)$, 
which is independent of the choice of $x$.

Some properties of the Gauss hypergeometric equation~\eqref{eqn:Gauss}
generalize to the $A$-hy\-per\-geo\-me\-tric setting. Of particular relevance is fact is that
if $\beta$ is sufficiently generic, the solutions of $D/\HA$ have
explicit combinatorial expressions, analogous
to~\eqref{eqn:GaussSolutions}, which are easily shown
to be holomorphic as functions of the parameters (see
Section~\ref{sec:LogFreeHolomorphic} for more details). We can apply
the method of parametric derivatives, which produced the 
solutions of~\eqref{eqn:Gauss} at special parameters, to the
$A$-hypergeometric situation. However, this method does
not produce all $A$-hy\-per\-geo\-me\-tric functions unless the holonomic rank
of $D/\HA$ is constant as a function of $\beta$~\cite[Section~3]{BFM}.
In general, however, the rank of $D/\HA$ is not
constant~\cite{MMW,berkesch}, and therefore the method of parametric
derivatives will not serve the purpose of this article. 

Our main tool to understand the parametric behavior of
$A$-hypergeometric functions is the fact that these functions can be
expanded as Nilsson series (convergent Puiseux series with
logarithms), in an algorithmic way. When 
$A$ is homogeneous, this follows from the fact that $D/\HA$ is a
regular holonomic $D$-module~\cite{hotta}, by the method of canonical
series~\cite[Sections~2.5 and~2.6]{SST}. If $A$ is not homogeneous,
then $D/\HA$ has irregular singularities~\cite{SW} (these systems are
also known as \emph{confluent}), but if the domain of
expansion is adequately chosen, one can still write $A$-hypergeometric
functions as Nilsson series~\cite{nilsson}.

When $A$-hypergeometric series have no logarithms, their parametric
behavior is understood (Section~\ref{sec:LogFreeHolomorphic}).
The core of our work is to understand the true logarithmic
$A$-hypergeometric series as functions of the parameters.
The key difficulty in this study is the fact that there are no general combinatorial
expressions for those solutions. When such combinatorial expressions
exist (for instance~\cite{adolphson-sperber}), our arguments may be simplified.

To compute the solutions of $D/H_A(\beta)$ as power series, we choose
special points (called \emph{toric infinities}) around which to expand. This is
equivalent to choosing a weight vector.
A vector $w \in \RR^n_{>0}$ is called a \emph{weight vector} for $A$
if it is sufficiently generic so that the initial ideal $\ini_w(I_A)$ with
respect to $w$ is a monomial ideal. In this case, $w$ induces a
\emph{coherent} (or \emph{regular}) \emph{triangulation} of
$\conv(A)$, the convex hull of the columns of $A$, by
projecting onto $\conv(A)$ the lower hull of the points $(w_i,a_i)$, where
$a_1,\dots,a_n$ are the columns of $A$. For
Theorems~\ref{thm:mainHomogeneous} and~\ref{thm:mainInhomogeneous}, we also require $w$
to be generic enough that the cones~\eqref{eqn:C}
and~\eqref{eqn:Cdual} are full dimensional.

Suppose that $A$ is homogeneous.
Let $w \in \RR_{>0}^n$ be a weight vector and $\Delta_w$ the induced
triangulation of $\conv(A)$. We define a stratification $\mathscr{S}$ of $\CC^d$
associated with the triangulation $\Delta_w$.

If $\sigma \subseteq \{ a_1,\dots,a_n\}$ is a codimension $1$ face of
$\Delta_w$, let $H_\sigma$ be the hyperplane spanned (as a vector space)
by the elements of $\sigma$. The \emph{codimension $0$ stratum of
  $\mathscr{S}$} is defined as 
\[
\mathscr{S}_0 \defeq
\CC^d \minus \bigg(\bigcup_{\substack{\sigma \in
  \Delta_w \\ \codim(\sigma) = 1}} \bigcup_{p\in \ZZ^d}
(p+H_\sigma)\bigg) .
\] 
Inductively, $\overline{\mathscr{S}_{i}} \minus
\mathscr{S}_{i}$ is an infinite, locally finite union of affine
spaces of codimension $i+1$.
(Here we use the closure in either the
Euclidean or the Zariski topology of $\CC^d$.) Then we define the
\emph{codimension $i+1$ stratum of $\mathscr{S}$} as
\[
\mathscr{S}_{i+1} \defeq
\bigcup_{H \in \overline{\mathscr{S}_{i}} \minus
\mathscr{S}_{i}} H \minus \bigg(\bigcup_{\substack{\sigma \in \Delta_w \\ \codim(\sigma)=1}}
\bigcup_{\substack{p \in \ZZ^d \\ p+H_\sigma \not \supseteq H }}
p+H_\sigma \bigg) ,
\]
where above, when we write $H \in \overline{\mathscr{S}_{i}} \minus
\mathscr{S}_{i}$, we mean that $H$ is an irreducible
component of $\overline{\mathscr{S}_{i}} \minus
\mathscr{S}_{i}$.

Note that the intersection of two irreducible components of $\overline{\mathscr{S}_{i}} \minus
\mathscr{S}_{i}$ is not contained in
$\mathscr{S}_{i+1}$ by construction. 
Consequently, if $q$ is a point in $\mathscr{S}_{i+1}$, then it belongs to a unique
irreducible component $H$ of $\overline{\mathscr{S}_i} \minus
\mathscr{S}_i$. 

The stratification $\mathscr{S}$ has the property that if an integer
translate of the span of a face of $\Delta_w$ intersects an irreducible
component of $\overline{\mathscr{S}_i} \minus \mathscr{S}_i$, then it
contains that whole irreducible component.

\begin{theorem}\label{thm:mainHomogeneous}
Assume that $A$ is homogeneous and $w\in \RR^n_{>0}$ is such that the
cones~\eqref{eqn:C} and~\eqref{eqn:Cdual} are full dimensional.
Let $\Delta_w$ be a coherent triangulation of $A$ arising from $w$, and consider the
stratification $\mathscr{S}$ of $\CC^d$ associated to $\Delta_w$ as above.
Then there is an
open set $U\subset\CC^n$ (depending only on $w$ and $A$) such that
the solutions of $\HA$ are holomorphic on $U\times \mathscr{S}_i$,  for
$i=0,\dots,d-1$.  
\end{theorem}

If $A$ is not homogeneous,
then a slightly weaker version of Theorem~\ref{thm:mainHomogeneous}
holds.
For $A$ a $d\times n$ integer matrix as above, we define $\rho(A)$ to
be the $(d+1)\times (n+1)$ matrix obtained from $A$ by appending $0
\in \ZZ^d$ on the left, and then appending $(1,\dots,1) \in\ZZ^{n+1}$
on top of the resulting $d\times(n+1)$ matrix.
In this case, we consider coherent triangulations $\Delta_w$ of $\conv(\{0,a_1,\dots,a_n\})$,
and construct the corresponding stratification $\mathscr{S}$ in the
same way as before. Moreover, only special triangulations can be used
in the following result.

\begin{theorem}\label{thm:mainInhomogeneous}
Let $A \in \ZZ^{d \times n}$, not necessarily homogeneous,
and let $w\in \RR^n_{>0}$  sufficiently generic that the
cones~\eqref{eqn:C} and~\eqref{eqn:Cdual} are full dimensional. 
Let $w \in \RR_{>0}^n$ be a weight vector
such that the closure of the cone~\eqref{eqn:C} contains
$(1,\dots,1) \in \RR^n$. Let $\Delta_w$ be the coherent triangulation of
$\conv(\{0, a_1,\dots,a_n\})$ arising from $(0,w)$, and let
$\mathscr{S}$ be the corresponding stratification of $\CC^d$. Then
there exists an open set $U \subset \CC^n$ (depending only on $w$ and $A$) such
that the solutions of $H_A(\beta)$ are holomorphic on $U \times
\mathscr{S}_i$ for $i=1,\dots,d-1$.
\end{theorem}

We remark that $A$-hypergeometric functions do not vary according to
isomorphism classes of 
$A$-hypergeometric systems, as computed in~\cite{saito-isom}.

\subsection{Relation to GG Functions}

In~\cite{GG97}, Gelfand and Graev introduced systems of
differential and difference equations, much generalizing the
hypergeometric systems studied here.
In a special case, these were also rediscovered by Ohara and Takayama~\cite{ohara-takayama}.
A key feature of the \emph{$GG$ systems} is that the parameters are now
considered as variables, and solutions are required to be holomorphic
in these new variables. However, the issues
addressed in this article do not arise in the $GG$ setting, because it no longer makes sense to
restrict to special parameters. For instance, the $GG$ system
for the Gauss hypergeometric equation~\eqref{eqn:Gauss} is
\begin{align*}
\frac{dF(z;a,b,c)}{dz} & = F(z;a+1,b+1,c+1),  \\
F(z;a+1,b,c) &= a F(z;a,b,c) + z F(z;a+1,b+1,c+1) ,\\
F(z;a,b+1,c) &= b F(z;a,b,c) + z F(z;a+1,b+1,c+1) ,\\
F(z;a,b,c-1) &= c F(z;a,b,c) + z F(z;a+1,b+1,c+1) ,
\end{align*}
whose solution space is spanned by 
\[
(1/\Gamma(c))F(z;a,b,c) \quad \text{and} \quad
(z^{1-c}/\Gamma(2-c))F(z;a+1-c,b+1-c,2-c), 
\]
where $\Gamma$ is the
Euler gamma function and $F$ is as
in~\eqref{eqn:GaussSolutions}. These series are entire in $a,b,c$, and 
holomorphic for $z$ in a punctured polydisc around $0 \in \CC$. Thus,
the Gauss $GG$ system does not capture the solutions of the Gauss
differential equation for $c\in \ZZ$. In general, the $GG$ systems
corresponding to the hypergeometric equations studied here give rise
to the solutions of those systems for generic parameters, as is 
stated in~\cite[Theorem~4]{GG99}.
See also Section~\ref{sec:LogFreeHolomorphic}.

\subsection*{Outline}
Sections~\ref{sec:SST}-\ref{sec:LogSeriesHolomorphic} lead to a proof of Theorem~\ref{thm:mainHomogeneous}. 
Section~\ref{sec:SST} contains background on canonical series
solutions, mostly from~\cite[Chapter~2]{SST}. 
Section~\ref{sec:LogFreeHolomorphic} introduces useful series of Horn
type and shows that logarithm-free
hypergeometric series are holomorphic in the parameters. 
Sections~\ref{sec:LogSeriesParametric}
and~\ref{sec:LogSeriesHolomorphic} extend this result to
hypergeometric series involving logarithms, culminating in a proof of Theorem~\ref{thm:mainHomogeneous}. 
Section~\ref{sec:confluentSolutions} gives an upper bound for rank in
the inhomogeneous case and provides the proof of
Theorem~\ref{thm:mainInhomogeneous}. 

\subsection*{Acknowledgements}
We are grateful to Alicia Dickenstein, Mar\'ia Cruz Fern\'andez
Fern\'andez, Pavel Kurasov, Ezra Miller, Timur Sadykov, and Uli
Walther for inspiring conversations and helpful comments provided
during the course of this work.

\section{Canonical series solutions of $A$-hypergeometric systems}
\label{sec:SST}

In this section, we summarize the essential properties of the
canonical series solutions from Chapter~2 of~\cite{SST} that are needed
in the sequel, and prove Theorem~\ref{thm:openSetInx}.

A left $D$-module of the form $D/I$ is said to be \emph{holonomic} if $\textrm{Ext}^i_D(D/I,D)=0$ for all $i \neq n$. 
An ideal $I$ is said to be holonomic if the quotient $D/I$ is a holonomic $D$-module.
In this case, by a result of Kashiwara (see ~\cite[Theorem~1.4.19]{SST})
the \emph{holonomic rank} of $D/I$, denoted by $\rank(D/I)$, which is by definition 
the dimension of the space of germs of holomorphic solutions of
$I$ near a generic nonsingular point of the system $I$, is finite.
By~\cite[Theorem~3.9]{adolphson}, the quotient $D/\HA$ is holonomic for all $\beta$,
and thus has finite rank.

The algorithm to compute canonical series solutions presented in~\cite[Section~2.6]{SST} applies to all \emph{regular holonomic} left $D$-ideals.
For an $A$-hypergeometric system, regular holonomicity is equivalent
to the matrix $A$ being homogeneous~\cite{hotta,SW}, an assumption
which we impose for the remainder of this section. 

Fix a weight vector $w \in \RR^n$. 
The $(-w,w)$-\emph{weight} of a monomial $x^u\del^v \in D$ is by definition
$ -w \cdot u + w \cdot v$. Thus, the weight $w$ induces a partial order on the set of monomials in $D$.
For a nontrivial $f(x, \del)=\sum_{u,v}c_{u v}x^u\partial^v \in D$, 
the \emph{initial form} $\inww(f)$ of $f$ with respect to $(-w,w)$
is the subsum of $f$ consisting of its (nonzero) terms of maximal $(-w,w)$-weight. 
If $I$ is a left $D$-ideal, its \emph{initial ideal} with respect
to $(-w,w)$ is
\[
\inww(I) = \< \inww(f) \mid f \in I, f \neq 0 \> \subset D.
\] 
Similarly, if $J \subseteq \CC[\del]$ is an ideal, we define the
initial ideal $\inw(J)$ as the ideal generated by the initial terms
of all nonzero polynomials in $J$.

Assume that $w$ is sufficiently generic so that the rational
polyhedral cones
  \begin{align}
  \label{eqn:C}
    C^{\phantom{*}} & \defeq  
    \{ w' \in \RR^n \mid \ini_{(-w',w')}
    (H_A(\bar{\beta}))=
    \inww(H_A(\bar{\beta})) \text{ for all } 
    \bar{\beta} \in \CC^d
    \}
\quad\text{and} \\
  \label{eqn:Cdual}
	C^* & \defeq \{ {\phantom{'}}v{\phantom{'}} \in \RR^n \mid v \cdot w' \geq 0 
	\text{ for all } w' \in C \}.
\end{align}
are full dimensional. Note that $C$ is open, while $C^*$ is closed
(with nonempty interior). 
Identify $C^*_\ZZ = C^*\cap\ZZ^n$ with the set $\{x^v\mid v\in C^*\cap\ZZ^n\}$ (notice
that this set depend on $w$), and set 
\[
N_w = \CC[[C^*_\ZZ]][\log(x)],
\] 
where $\log(x) = (\log(x_1),\dots,\log(x_n))$. 
We call an element of $N_w$ of the form $\log(x)^\delta$ a \emph{logarithmic monomial},
and an element of the form $x^\gamma\log(x)^\delta$ a \emph{mixed monomial}.
A series $\varphi(x) = \sum_{\gamma,\delta} c_{\gamma
  \delta}\,x^\gamma \log(x)^\delta\in N_w$ is called a
logarithmic series. 
By~\cite[Proposition~2.5.2]{SST}, if $\varphi$ is a
logarithmic series which is a solution to $H_A(\beta)$ then, with
notation as above, the set of real parts 
$\{\textrm{Re}(\gamma \cdot w) \mid  c_{\gamma \delta} \neq 0 \text{
  for some } \delta \}$
achieves a (finite) minimum which is denoted $\mu(\varphi)$.
Furthermore, the subseries of $\varphi$ whose 
terms are $c_{\gamma \delta} x^{\gamma} \log(x)^\delta$ such that
$c_{\gamma \delta} \neq 0$ and
$\textrm{Re}(\gamma \cdot w)  = \mu(\varphi)$ is finitely supported. 
We call this finite sum the
\emph{initial series of $\varphi$ with respect to $w$} and we denote it by
$\inw(\varphi)$. 

\begin{definition}
\label{def:canonicalSeries}
A \emph{canonical series solution of $\HA$ (with respect to $w$)} is a
series $\varphi\in N_w$ such that 
\begin{enumerate}
\item 
 $\varphi$ is a \emph{formal solution} of $\HA$,
\item the initial series $\inw(\varphi)$ is a mixed monomial, 
\item the mixed monomial $\inw(\varphi)$ is the unique element of the set
\begin{align}\label{eqn:Start}
 \{\inw(\varphi)\mid \varphi 
  \in N_w \text{ is a (formal) solution of $\HA$}\}
\end{align}
appearing in the logarithmic series $\varphi$ with nonzero coefficient. 
\end{enumerate}
Let $\varphi$ be a canonical series solution of $\HA$ with respect to
$w$, and write $\inw(\varphi) = x^\gamma(\log(x))^\delta$. The vector
$\gamma$ is called \emph{exponent} of $\varphi$. The \emph{exponents}
of $\HA$ with respect to $w$ are the exponents of all its canonical
series with respect to $w$.
\end{definition}

\begin{theorem}[{\cite[Theorem 2.5.16]{SST}}]
\label{thm:css}
Let $A$ be homogeneous, $\beta \in \CC^d$ and $w$ a weight vector. 
Then there exists a domain $W\subseteq\CC^n$ (depending on $w$) such that, for any $x\in W$, 
the canonical series solutions of $\HA$ form a basis of\/ $\Sol_x(\HA)$. 
\qed
\end{theorem}

In particular, canonical series solutions exist.
Note also that, with notation as in Theorem~\ref{thm:css}, for each
$x\in U$ there are precisely $\rank(\HA)$-many canonical series
solutions of $\HA$ that converge in a neighborhood of $x$ and are
linearly independent over $\CC$.

Recall that a weight vector also gives rise to a triangulation
$\Delta_w$ of $\conv(A)$. The following fact, linking canonical series
solutions and triangulations, is a consequence
of~\cite[Lemma~4.1.3]{SST} and~\cite[Equation~3.7]{SST}.

\begin{lemma}
\label{lemma:exponentSupport}
Let $\gamma$ be an exponent of $\HA$ with respect to $w$. Then the set
$\{i \in \{1,\dots,n\} \mid \gamma_i \notin \NN \}$ are the vertices
of a simplex of the triangulation $\Delta_w$, and the columns of $A$
indexed by this set are linearly independent.
\qed
\end{lemma}

The following result provides a common domain of convergence in $x$ for
all canonical series solutions with respect to a weight vector
$w$, regardless of the parameter.

\begin{theorem}
\label{thm:openSetInx}
Let $w\in \RR_{>0}^n$ be a weight vector. 
Then there exists an open, nonempty subset $V \subseteq \CC^n$ such that for
all $\beta \in \CC^d$, the canonical series solutions of $\HA$
with respect to $w$ are absolutely convergent for $x \in V$.
\end{theorem}

\begin{proof}
Consider $D[\beta]$, the Weyl algebra with $d$ commuting
indeterminates $\beta_1,\dots, \beta_d$ adjoined.  
By~\cite{kredel:weispfenning,weispfenning}, there exists a
\emph{comprehensive Gr\"obner basis} for $\HA$ with respect to $(-w,w)$. 
This is a finite subset $G$ of $D[\beta]$ whose specialization for any
fixed $\bar{\beta}\in\CC^d$ is a Gr{\"o}bner basis of
$H_A(\bar{\beta})$ with respect to $(-w,w)$.  
Consider an element $P(\beta,x,\del)
= \sum q_{u,v}(\beta) x^u \del^v$ of $G$. Given $\bar{\beta}\in 
\CC^d$, the support of the initial form $\inww(P(\bar{\beta},x,\del))$
with respect to $x$ and $\del$, depends on the (vanishing of) the coefficients $q_{u,v}(\bar{\beta})$. 
For each possible support of a specialization of $P$ at
$\bar{\beta}$, the conditions on $w' \in \RR^n$ so that 
$\ini_{(-w',w')}(P) = \inww(P)$ are linear equations and inequalities, and
therefore the set, 
\[
\{ 
  w' \in \RR^n \mid
  \ini_{(-w',w')}(P(\bar{\beta},x,\del))=\inww(P(\bar{\beta},x,\del))
  \text{ for all } \bar{\beta} \in \CC^d \}
\]
is (in the Euclidean topology) a relatively open rational polyhedral cone. 
Applying the same argument to all elements of $G$, we conclude that 
the cone $C$ from \eqref{eqn:C} is a relatively open rational polyhedral cone. 
If $w$ is sufficiently generic, this cone is full dimensional, and therefore
open in $\CC^n$.
Now~\cite[Theorem~2.5.16]{SST} implies that there exists $\bar{x} \in C$
such that every canonical series solution of $H_A(\bar{\beta})$ with respect
to $w$ converges absolutely for $(-\log|x_1|,\dots,-\log|x_n|) \in
\bar{x}+C$, and this holds for all $\bar{\beta} \in  
\CC^d$. 
\end{proof}

\section{Logarithm-free hypergeometric series are holomorphic in the parameters}
\label{sec:LogFreeHolomorphic}

In this section, we assume that $A$ is homogeneous, and summarize
known results about $A$-hy\-per\-geome\-tric series without logarithms.

A series $\varphi \in N_w$ is \emph{logarithm-free} if it contains no
term $x^\gamma \log(x)^\delta$ with $\delta \neq 0$. A detailed study
of such series and their exponents can be found
in~\cite[Section~3.4]{SST}.

The exponents of logarithm-free $A$-hypergeometric series have
\emph{minimal negative support}. This means that, if $\bar\alpha$ is such 
an exponent, $\{ i \mid \bar\alpha_i \in \ZZ_{<0} \} \subseteq \{i \mid
(\bar\alpha+u)_i \in \ZZ_{<0} \}$ for all $u \in \ker_\ZZ(A)$.
Conversely, if $\bar\alpha$ is an exponent of $H_A(\bar\beta)$ with respect to
$w$ that has minimal negative support, then it is the exponent of a
(unique) logarithm-free canonical series solution of $H_A(\bar\beta)$ (note that
different canonical series solutions of $H_A(\bar\beta)$ may have the same
exponent). See~\cite[Theorem~3.4.14, Corollary~3.4.15]{SST}.

More precisely, if $\bar\alpha$ is an exponent of $H_A(\bar\beta)$ with minimal
negative support, and 
\begin{equation}
\label{eqn:support}
S = \big\{ u \in \ker_\ZZ(A) \mid \{ i \mid \bar\alpha_i \in \ZZ_{<0 }\} = \{ i \mid
(\bar\alpha+u)_i \in \ZZ_{<0}\} \big\},
\end{equation}
then
\[
\sum_{u\in S} 
\frac{\prod_{u_i<0}
  \prod_{j=1}^{-u_i} (\bar\alpha_i-j+1)}{\prod_{u_i>0}
  \prod_{j=1}^{u_i} (\bar\alpha_i+j)}\, x^{\bar\alpha+u}
\]
is the unique logarithm-free canonical series solution of $H_A(\bar\beta)$ with
exponent $\bar\alpha$~\cite[Proposition~3.4.13]{SST}. We remark that $S \subset C^* \cap \ker_\ZZ(A)$,
where $C^*$ is the cone from~\eqref{eqn:Cdual} (see the proof of~\cite[Theorem~3.4.14]{SST}). 

Since the expression above is so explicit, we may use it to view
logarithm-free $A$-hypergeometric 
series not only as a functions of
$x$ but also as functions of (some of the coordinates of) the corresponding
exponent, and consequently, as functions of the parameters.

Let $\bar\alpha$ be an exponent of $H_A(\bar\beta)$ with respect to a weight
vector $w$, and assume that $\bar\alpha$ has minimal negative
support. Let $\sigma = \{ i \mid \gamma_i \notin \ZZ\}$.
(In this
case, $\sigma$ is the set of vertices of a simplex in the triangulation $\Delta_w$ by
Lemma~\ref{lemma:exponentSupport}.) 

Now let $\alpha_j \in \CC\minus \ZZ$ for $j\in \sigma$, and denote
$\alpha \in \CC^n$ the vector whose coordinates indexed by $\sigma$
are  the $\alpha_i$, and whose remaining coordinates are the $\bar\alpha_i$. 
Then $\alpha \in (\CC\minus \ZZ)^{\sigma} \times \{
(\bar\alpha_i)_{i\notin \sigma} \}$. In what follows, we abuse
notation and write $\alpha \in (\CC\minus \ZZ)^{\sigma}$, as the
other coordinates of $\alpha$ are fixed.
It can be shown~\cite[Corollary~3.4.15]{SST} that $\alpha$ is an
exponent of $H_A(A \cdot \alpha)$ with minimal negative support,
whose corresponding logarithm free $A$-hypergeometric series is

\begin{equation}
\label{eqn:seriesFor-alpha}
\varphi(x;\alpha) = \sum_{u\in 
S}
\frac{\prod_{u_i<0}
  \prod_{j=1}^{-u_i} (\alpha_i-j+1)}{\prod_{u_i>0}
  \prod_{j=1}^{u_i} (\alpha_i+j)}\, x^{\alpha+u},
\end{equation}
where the sum is over the same set $S$ as in~\eqref{eqn:support}. As
this is a canonical series solution of a hypergeometric system, every
time we fix
$\alpha \in (\CC\minus \ZZ)^\sigma$, the series $\varphi(x;\alpha)$
converges on the open set $V$ from Theorem~\ref{thm:openSetInx}.

Thus we see that we may consider $\varphi(x;\alpha)$ as a function both
of $x$ and $\alpha$. The parameter $A \cdot \alpha$ is
a function of $\alpha$; on the other hand, we may also consider $\alpha$ as a
function of the parameter $A\cdot \alpha$, because the columns of $A$
indexed by $\sigma$ are linearly independent by Lemma~\ref{lemma:exponentSupport}.

The following result gives precise information on the convergence of $\varphi(x;\alpha)$.

\begin{theorem}
\label{thm:logFreeSeriesHolomorphic}
The series
$\varphi(x;\alpha)$ is holomorphic for $(x,\alpha) \in V\times (\CC
\minus \ZZ)^\sigma$, where
$V \subset \CC^n$ is the open set from Theorem~\ref{thm:openSetInx}.
\end{theorem}

\begin{proof}
This is a well known result (see, for
instance~\cite[Lemma~1]{ohara-takayama}
and~\cite[Theorem~4]{GG99}). We sketch the proof here for the 
readers' convenience. 

Let $\Gamma$ denote the Euler gamma function. Then 
\begin{equation}
\label{eqn:gammaSeries}
\frac{1}{\prod_{i=1}^n \Gamma(\alpha_i+1)} \ \varphi(x;\alpha) = 
\sum_{u \in S} \frac{1}{\prod_{i=1}^n \Gamma(\alpha_i+u_i+1)} \ x^{\alpha+u}.
\end{equation}

Multiply and divide on the right hand side to obtain
\[
\sum_{u\in S} \prod_{j \in \sigma }\frac{\Gamma(\bar\alpha_j + u_j+1)
}{\Gamma(\alpha_j+u_j+1) } \
\frac{1}{\prod_{i=1}^n \Gamma(\bar\alpha_i + u_i+1)} \ x^{\alpha+u},
\]
and consider the factor $\prod_{j \in \sigma }\frac{\Gamma(\bar\alpha_j + u_j+1)
}{\Gamma(\alpha_j+u_j+1) }$. 
Fix a compact set $K\subset (\CC\minus \ZZ)^\sigma$ and suppose
$\alpha \in K$.
Since the elements $u\in S$ lie in the pointed
cone $C^*$ from~\eqref{eqn:Cdual}, the behavior of $u_j$ for $j\in \sigma$ is constrained:
either $|u_j|$ is bounded for all $u \in S$, or $u_j$ is bounded below
and $u_j \to \infty$, or $u_j$ is bounded above and 
$u_j \to -\infty$. For $\alpha \in K$ and sufficiently large $t$,
$\Gamma(\bar\alpha_j+t+1)/\Gamma(\alpha_j+t+1)$ is approximately one, for
instance, by Stirling's approximation. Therefore, the series on the right hand
side of~\eqref{eqn:gammaSeries} can be bounded term by term in
absolute value by the series $\sum_{u\in S} x^{\alpha+u} /
\prod_{i=1}^n \Gamma(\bar\alpha_i+u_i+1)$, and this has the same domain of
convergence in $x$ as the series $\sum_{u\in S} x^{\bar\alpha+u} /
\prod_{i=1}^n \Gamma(\bar\alpha_i+u_i+1)$; recall that this domain $V$ was
found in Theorem~\ref{thm:openSetInx}. 

Moreover, this argument shows that for each fixed $x \in V$, the
partial sums of~\eqref{eqn:gammaSeries} 
converge absolutely and uniformly for $\alpha$ on
compact subsets of $(\CC\minus \ZZ)^\sigma$. As each partial sum is
holomorphic in $\alpha$, a standard application of Morera's theorem
yields that the limit is holomorphic in $\alpha$ as well.
\end{proof}

We remark that if $\beta \in \mathscr{S}_0$, where $\mathscr{S}$ is
the stratification associated to the triangulation $\Delta_w$
constructed in Section~\ref{sec:intro}, all solutions of
$D/\HA$ with respect to $w$ are logarithm-free~\cite[Proposition~3.4.4
and~Theorem~3.4.14]{SST}. Consequently,
Theorem~\ref{thm:logFreeSeriesHolomorphic} proves
Theorem~\ref{thm:mainHomogeneous} for the codimension zero stratum of $\mathscr{S}$.

\subsection{Series of Horn type}
\label{ssec:HornSeries}

In this subsection we consider logarithm free series of Horn type,
which are related to $A$-hypergeometric series. Horn series are a key
ingredient in proving the results of Section~\ref{sec:LogSeriesHolomorphic}.

Let $B_1,\dots,B_n \in \ZZ^m$ denote the rows of a rank $m$, $n\times m$ matrix
$B=[b_{jk}]$ such that $A \cdot B = 0$, where $m=n-d$. 
For $k=1,\dots, m$, let us define polynomials in the variables $\mu =
(\mu_1,\dots,\mu_m)$ and parameters $(\alpha_1,\dots,\alpha_n)$ by 
\begin{equation}
\label{eqn:HornEquations}
P_k(\mu; \alpha) 
 = \prod_{b_{jk}>0} \prod_{\ell=0}^{b_{jk}-1} 
  (B_j \cdot \mu + \alpha_j - \ell)
\quad\text{and}\quad 
Q_k(\mu; \alpha) 
  = \prod_{b_{jk}<0} \prod_{\ell=0}^{|b_{jk}|-1}
  (B_j \cdot \mu + \alpha_j - \ell).
\end{equation}

For $\mu \in \NN^n$, let 

\begin{equation}
\label{eqn:hornSeries}
R_\mu(\alpha) =   \prod_{\ell=1}^m
  \prod_{j_\ell=0}^{\mu_\ell-1} \frac{
  P_\ell(\mu_1,\dots,\mu_{\ell-1},
  j_\ell,0,\dots,0)}{
  Q_\ell(\mu_1,\dots,\mu_{\ell-1},
  j_\ell+1,0,\dots,0)} \; ;
\quad 
\Phi_B(z;\alpha) =
  \sum_{\mu \in \NN^n} R_\mu(\alpha) z^\mu.
\end{equation}

Note that the $R_\mu(\alpha)$ are rational functions in $\alpha$. For
$\mu \in \NN^n$, the poles of $R_\mu(\alpha)$ occur when some
coordinates of $\alpha$ take specific integer values.

The proof of Theorem~\ref{thm:logFreeSeriesHolomorphic} applies to
show the following.

\begin{theorem}
\label{thm:HornSeriesConverge}
Let $Y=\{ \alpha \in \CC^n \mid  R_\mu(\alpha) \text{ has no poles }
\forall \mu \in \NN^n \} \supset (\CC \minus \ZZ)^n$. There exists a polydisc $W \subset \CC^m$
centered at the origin such that $\Phi_B(z;\alpha)$ is holomorphic for
$(z,\alpha) \in W \times Y$. The polydisc $W$ depends only on the
matrix $B$, and not on the parameters $\alpha$.
\qed
\end{theorem}

\section{Logarithmic $A$-hypergeometric series are continuous in the parameters}
\label{sec:LogSeriesParametric}

$A$ is assumed to be homogeneous throughout this section.
We already know from the previous section that a canonical
logarithm free $A$-hy\-per\-geo\-me\-tric series can be regarded as a function
of its parameters.  
Since we lack such explicit expressions for logarithmic
$A$-hypergeometric series in general (\cite{adolphson-sperber}
provides combinatorial expressions for logarithmic $A$-hypergeometric
series under special conditions), 
it is nontrivial to see that the  
coefficients in such a series depend continuously on the parameters.
The goal of this section is to prove Theorem~\ref{thm:perturbParameterLogSeries}, 
which implies that for canonical logarithmic $A$-hypergeometric series
this dependence is indeed continuous.
We begin by describing the general form of such a series.

Let $w\in\RR_{>0}^n$ be a weight vector, 
and consider the cones $C$ and $C^*$ from~\eqref{eqn:C} and \eqref{eqn:Cdual}. 
Since $w$ is generic, 
$C^*$ is a convex, pointed rational polyhedral cone with nonempty interior.

Let $\bar \beta \in \CC^d$, and let $\bar \alpha$ be an exponent of
$H_A(\bar{\beta})$ with respect to the weight vector $w$. 
In particular, $A\cdot \bar \alpha = \bar \beta$. We can write
the canonical series solution of $H_A(\bar{\beta})$ corresponding to
$\bar \alpha$ as follows;
\begin{align}\label{eqn:barvarphi}
\varphi(x; \bar\alpha) = x^{\bar{\alpha}} \sum_{u \in C^* \cap \ker_{\ZZ}(A)} 
  p_u(\log(x); \bar \alpha)\, x^u 
\end{align}
where  $p_u(\log(x); \bar\alpha)$ is a polynomial
in the $n$ variables $\log(x_1),\dots,\log(x_n)$ (whose coefficients
depend on $\bar \alpha$).
It is known that the degree in $\log(x)$ of the polynomial $p_u$ is bounded by $n$
times the rank of $D/H_A(\bar{\beta})$
(see~\cite[Theorem~2.5.14]{SST}). We remark that
$\rank(D/H_A(\bar{\beta})) \leq 2^{2d} \vol(A)$
by~\cite[Corollary~4.1.2]{SST}; this gives the parameter-independent
bound $n 2^{2d} \vol(A)$ for the degree of the polynomials $p_u$.

Since $C$ is open and $w\in C$, for any $T>0$, 
the set $\{ u \in C^* \mid u \cdot w \leq T \}$ is bounded.
The set $\{ u \in C^*\cap\ker_{\ZZ}(A) \mid p_u \neq 0\}$ 
is called the \emph{support} of $\varphi$.

Assume that $\bar\alpha_1 \notin \ZZ$. In the following result, 
we use $\alpha$ to denote a vector all of whose coordinates coincide
with those of $\bar\alpha$ except possibly the first one.

\begin{theorem}
\label{thm:perturbParameterLogSeries}
Fix a weight vector $w$, and parameter $\bar{\beta}\in\CC^d$.
Let $\varphi(x;\bar\alpha)$ be the logarithmic canonical series solution~\eqref{eqn:barvarphi}.
Suppose $\bar{\alpha}_1 \notin \ZZ$. Then there exist a neighborhood $W$
of $\bar{\alpha}_1$, and $p_u(y_1,\dots,y_n;\alpha) \in \CC(\alpha_1) [y_1,\dots,y_n]$
for $u \in C^* \cap \ker_\ZZ(A)$, such that for $\alpha_1 \in W$, the series
$ \varphi(x;\alpha) = x^\alpha \sum_{u \in C^* \cap \ker_\ZZ(A)}
p_u(\log(x);\alpha) x^u$ is a canonical series solution of
$H_A(A\cdot\alpha)$ with respect to $w$ (and therefore converges for
$x\in V$, where $V$ is the open set from
Theorem~\ref{thm:openSetInx}).
Moreover when
$\alpha$ is set to $\bar\alpha$, $\varphi(x;\alpha)$ specializes to ~\eqref{eqn:barvarphi}.
\end{theorem}

To prepare for the proof of Theorem~\ref{thm:perturbParameterLogSeries}, we prove the following
technical result;
while it is easy to see that the derivative of an $A$-hypergeometric
series is $A$-hypergeometric (with a shifted parameter), we provide a
sufficient condition for a possibly logarithmic $A$-hypergeometric
series to have an $A$-hypergeometric antiderivative. 

\begin{proposition}
\label{prop:antiderivative}
Let $\varphi(x; \bar\alpha)$ be a (logarithmic) canonical series
solution of $H_A(\bar{\beta})$ with respect to a weight vector $w$, 
and let $a_1$ denote the first column of $A$. 
If $\bar{\alpha}_1 \notin \ZZ$, 
then there exists a canonical solution $\psi$ of
$H_A(\bar{\beta}+a_1)$ with respect to $w$ such that $\del_1\psi =
\varphi$. 
\end{proposition}

\begin{lemma}{\cite[Lemma~3.12]{nilsson}}
\label{lemma:antilog}
Let $p$ be a polynomial in $n$ variables 
and $\bar{\alpha} \in \CC^n$ with $\bar{\alpha}_1\neq -1$. 
Then, there exists a unique polynomial $q$ in $n$ variables such that $\deg(q)=\deg(p)$ and 
\[
\hspace{153pt}
\del_1 x_1 x^{\bar{\alpha}}\, q(\log(x)) = x^{\bar{\alpha}}\,p(\log(x)).
\hspace{145pt}\square
\]
\end{lemma}

With the notation of Lemma~\ref{lemma:antilog}, set
\[
\del_1^{-1} \big[x^{\bar{\alpha}}\, p(\log(x)) \big] 
  \defeq x_1\,x^{\bar{\alpha}}\, q(\log(x)).
\]

The next two results follow from the uniqueness in Lemma~\ref{lemma:antilog}.

\begin{lemma}
\label{lemma:fixEulers}
Let $p$ be a polynomial in $n$ variables 
and $\bar{\alpha} \in \CC^n$ with $\bar{\alpha}_1\neq -1$. 
If $x^{\bar{\alpha}}\,p(\log(x))$ is a solution of 
$\<E-\bar{\beta}\>$, 
then $\del_1^{-1}\,x^{\bar{\alpha}} \,p(\log(x))$ is a solution of 
$\<E-(\beta+a_1)\>$. 
\qed
\end{lemma}

\begin{lemma}
\label{lemma:commutingDerivatives}
Let $p$ be a polynomial in $n$ variables 
and $\bar{\alpha}\in \CC^n$ with $\bar{\alpha}_1\neq -1$. 
Then for any $\mu \in \NN^n$,
\[
\hspace{121pt}
\del^\mu \big[ \del_1^{-1}\, x^{\bar{\alpha}}\,
  p(\log(x)) \big] 
  = \del_1^{-1} \big[
  \del^\mu\, x^{\bar{\alpha}}\, p(\log(x)) \big].
\hspace{114pt}\square
\]
\end{lemma}

\begin{proof}[Proof of Proposition~\ref{prop:antiderivative}]
Since $\bar{\alpha}_1 \notin \ZZ$, the first coordinate of $u+\bar{\alpha}$ never equals $-1$.
We may define a formal power series 
\[\
\psi(x, \bar\alpha)
  \ = \ 
  \sum_{u\in C^* \cap \ker_{\ZZ}(A)} \del_1^{-1}
  \big[p_u(\log(x); \bar \alpha)\,x^{u+\bar{\alpha}}\big].
\] 
By Lemmas~\ref{lemma:fixEulers} and~\ref{lemma:commutingDerivatives}, 
the series $\psi$ is a solution of 
$H_A(\bar{\beta}+a_1)$ and $\del_1\psi=\varphi$. 
From this equality, we also conclude that 
$\psi$ is a canonical series solution of $H_A(\bar{\beta}+a_1)$.
In particular, $\psi$ is convergent. 
\end{proof}

\begin{proof}[Proof of Theorem~\ref{thm:perturbParameterLogSeries}]
Our first task is to show that for $\alpha_1$ in a neighborhood of
$\bar\alpha_1$, $\alpha$ is an exponent of $H_A(A\cdot \alpha)$ with respect
to $w$.
Let $G$ be a comprehensive Gr\"obner basis of $\HA$ with respect to
$(-w,w)$, 
and restrict to parameters of the form $\beta=A\cdot \alpha$ with
$\alpha_j=\bar{\alpha}_j$ for $j\neq 1$, so that the elements of $G$
are  
polynomials whose monomials, in the variables
$x$ and $\del$, have coefficients that are polynomial in $\alpha_1$. 
Recall that the degrees of the logarithmic polynomial coefficients in
an $A$-hypergeometric function are bounded by $n 2^{2d}\vol(A)$.
We may assume that there exists some $u\in C^*\cap \ker_\ZZ(A)$ such that
the polynomial $p_u(\log(x);\bar{\alpha})$ has positive degree.  

We have that $\del_1^k \varphi$ is a solution of $H_A(\bar{\beta}-ka_1)$ for all $k \in \NN$. 
Moreover, by repeated application of 
Proposition~\ref{prop:antiderivative}, 
for any $k\in\NN$, there is a solution 
$\psi_k$ of $H_A(\bar{\beta}+ka_1)$
such that $\del_1^k \psi_k = \varphi$. 
All of these hypergeometric series have the same support as $\varphi$,
and even the degrees of the logarithmic polynomial coefficients are
preserved. 

Since univariate polynomials have finitely many roots, the polynomials in $\alpha_1$ that are the coefficients of the elements of $G$ do not vanish for $\alpha_1=\bar{\alpha}_1+k$ when $k\in \NN$ is sufficiently large. 
Fix such a $k$. 
Once the result is proven for $\bar{\alpha}+ka_1$, applying $\del_1^k$ yields the desired result. 
Thus, we may assume that all the polynomial coefficients in $G$ are nonzero at $\bar{\alpha}_1$.

The conditions for $x^\alpha \,p_0(\log(x); \alpha)$ 
(where the degree of $p_0$ as a polynomial in $\log(x)$ does not depend on $\alpha$) 
to be a solution of $\inww(H_A(A \alpha))$ for $\alpha$ in a neighborhood of $\bar{\alpha}$ 
is a linear system of equations whose augmented matrix has entries
that are polynomial in $\alpha_1$. 

The solvability of a system of linear equations can be expressed as a
rank condition on its augmented matrix.
Thus, we have a solution for $\alpha$ in some algebraic subvariety $Y\subset \CC$, defined by 
certain minors of the augmented matrix.
By hypothesis, the required conditions are satisfied at $\bar{\alpha_1}$, 
hence $Y$ is nonempty. 
Actually, there exists a solution for $\alpha_1 = \bar{\alpha}_1 \pm \ell$ if $\ell$ is a sufficiently large positive integer, implying that $Y$ is of infinite cardinality. We conclude that $Y = \CC$.
Hence, there exists $\varepsilon >0$ such that for all $\alpha$ with
$|\alpha_1-\bar{\alpha}_1|<\varepsilon$ (and $\alpha_j = \bar \alpha_j$ for $j\neq 1$)
the system $\inww(H_A(A\cdot\alpha))$ has a solution of the form $x^\alpha\, p_0(\log(x); \alpha)$, 
where the degree of $p_0$ as a polynomial in $\log(x)$ does not depend
on $\alpha$. In particular, $\alpha$ is an exponent of $H_A(A\cdot\alpha)$
with respect to $w$ by~\cite[Lemma~2.5.10, Corollary~2.5.11]{SST}.

Now, for all $\alpha$ with $|\alpha_1-\bar{\alpha}_1|<\varepsilon$, 
we must construct a (canonical) series solution $\varphi$ of
$H_A(A\cdot\alpha)$ with respect to $(-w,w)$ of the
form~\eqref{eqn:barvarphi}.  
We follow the algorithm to compute canonical series solutions that appears in~\cite[Section~2.5]{SST}. 
Given $u\in C^* \cap \ker_{\ZZ}(A)$, this algorithm allows us to set
up a system of linear equations involving the coefficients of $p_u$
and those of $p_v$, for any  
\begin{align}\label{eqn:possiblev}
v\in C^* \cap \ker_{\ZZ}(A) 
\quad \text{such that} \quad 
w\cdot v \leq w \cdot u.
\end{align}
Solving this system we obtain the polynomial $p_u$.
This system has finitely many equations, since there are only finitely many elements 
that satisfy~\eqref{eqn:possiblev}. 
Moreover, the coefficients 
of the augmented matrix
of the linear system depend polynomially on $\alpha_1$.
If $\ell$ is a sufficiently large positive integer, the polynomials whose nonvanishing is a condition for the solvability of the system under consideration have no roots when
$|\alpha_1-(\bar{\alpha}_1+\ell)|<\varepsilon$. 
The same argument as above implies that the polynomials that are
required to vanish for the system to be solvable are identically zero.  
Upon obtaining a solution
$x^{u+\alpha}p_u(\log(x);\alpha)$ that is valid for
$|\alpha_1-(\bar{\alpha}_1+\ell)|<\varepsilon$, 
the polynomial $\del_1^{\ell}x^{u+\alpha}p_u(\log(x);\alpha)$ 
provides a solution that is valid for
$|\alpha_1-\bar{\alpha}_1|<\varepsilon$. 
To see this, two ingredients are necessary. First,
that $\del_1^{\ell}$ transforms solutions of
$\<E-A \alpha\>$ into solutions of
$\<E-A \alpha-\ell a_1\>$, and second, that
the toric operators have constant coefficients and therefore commute
with $\del_1^k$.

Finally, since there is a unique (up to a constant multiple) canonical
solution of $H_A(A\cdot\alpha)$ corresponding 
to the starting term $x^\alpha_0(\log(x);\alpha)$, the fact that the
coefficients of the polynomials $p_u(\log(x);\alpha)$ are determined
by solving linear systems whose augmented matrices depend polynomially
on $\alpha_1$ implies that the coefficients of $p_u(\log(x);\alpha)$ are rational functions
of $\alpha_1$.
\end{proof}

In the proof of Theorem~\ref{thm:perturbParameterLogSeries}, we
see that we can extend $\varphi(x;\bar\alpha)$ not just to $\alpha_1$ in a neighborhood of
$\bar\alpha_1$, but to $\alpha_1 \in \CC\minus \ZZ$, as all the
quantities involved are rational functions of $\alpha_1$, and $\alpha_1$ only
needs to avoid the poles of those rational functions.
Moreover, although only one parameter was perturbed in Theorem~\ref{thm:perturbParameterLogSeries}, 
it is possible to perturb any noninteger coordinate of the exponent
$\bar{\alpha}$, and obtain that the desired coefficients are
rational functions of the perturbed coordinates. 

\begin{corollary}
\label{coro:logSeriesInParameters}
Let $\bar{\alpha}$ be an exponent of $H_A(\bar{\beta})$ 
whose corresponding solution $\varphi(x;\bar\alpha)$, as in~\eqref{eqn:barvarphi}, is logarithmic.
Let $\sigma = \{i \mid \bar\alpha_i \notin \ZZ\}$. Let $\alpha$ denote
a vector whose coordinates indexed by $\sigma$ are $\alpha_i\in
\CC\minus \ZZ$, and whose coordinates not indexed by $\sigma$ coincide
with those of $\bar\alpha$.
Then there exist
$p_u(y_1,\dots,y_n;\alpha) \in \CC(\alpha_i \mid i \in \sigma) [y_1,\dots,y_n]$
for $u \in C^* \cap \ker_\ZZ(A)$, such that for $\alpha \in (\CC\minus
\ZZ)^\sigma$, the series
$ \varphi(x;\alpha) = x^\alpha \sum_{u \in C^* \cap \ker_\ZZ(A)}
p_u(\log(x);\alpha) x^u$ is a canonical series solution of
$H_A(A\cdot\alpha)$ with respect to $w$ (and therefore converges for
$x\in V$, where $V$ is the open set from
Theorem~\ref{thm:openSetInx}). Moreover when
$\alpha$ is set to $\bar\alpha$, $\varphi(x;\alpha)$ specializes to ~\eqref{eqn:barvarphi}.
\qed
\end{corollary}

\section{Logarithmic $A$-hypergeometric series are holomorphic in the parameters}
\label{sec:LogSeriesHolomorphic}

In this section, we again assume $A$ is homogeneous. 
We show that logarithmic $A$-hy\-per\-geo\-me\-tric series are
hol\-o\-mor\-phic functions of the parameters,
then we prove Theorem~\ref{thm:mainHomogeneous}.

Let $\bar{\alpha}$ be an exponent of $H_A(\bar{\beta})$ 
whose corresponding solution $\varphi(x,\bar\alpha)$, as
in~\eqref{eqn:barvarphi}, is logarithmic.

\begin{theorem}
\label{thm:logSeriesHolomorphic}
Resume the hypotheses and notation from
Corollary~\ref{coro:logSeriesInParameters}. There exists an open set
$U \subseteq V \subset \CC^n$ (where $V$ is from
Theorem~\ref{thm:openSetInx}) such that 
$\varphi(x;\alpha)$ is holomorphic for $(x,\alpha) \in U \times
(\CC\minus \ZZ)^\sigma$. The open set $U$ does not depend on the
logarithmic $A$-hypergeometric series~\eqref{eqn:barvarphi}, on the
set $\sigma$, or on the parameters $\alpha$.
\end{theorem}

We prove Theorem~\ref{thm:logSeriesHolomorphic} through an inductive
procedure, based on the following result.  
Recall that we use the convention that $0\in\NN$. 

\begin{lemma}
\label{lemma:maximalLogTerm}
Rewrite a solution $\varphi$ of $H_A(\bar{\beta})$ as in~\eqref{eqn:barvarphi} as 
\[
\varphi(x,\bar\alpha)
  = \sum_{\gamma \in \NN^n} \varphi_\gamma(x,\bar\alpha) 
  \log(x)^{\gamma},
\]
where $\varphi_\gamma$ are logarithm-free series. We refer to
$\gamma$ as the \emph{logarithmic exponent} of $\varphi_\gamma
\log(x)^\gamma$.
If $\delta$ is a componentwise maximal element of 
$\{\gamma \in \NN^n \mid \varphi_{\gamma} \neq 0 \}$, 
then $\varphi_{\delta}$ is a (logarithm-free) 
solution of $H_A(\bar{\beta})$.
\end{lemma}

\begin{proof}
If $P$ is a differential operator, then 
\[
P \bullet \big[ 
  \varphi_\gamma \log(x)^\gamma \big] 
  = \log(x)^\gamma \big[ P \bullet 
  \varphi_\gamma \big ] + 
  \text{ terms with lower logarithmic exponent}.
\]
Thus, if $P$ annihilates $\varphi$, then $P$ must annihilate $\varphi_\delta$ for any componentwise
maximal $\delta$.
\end{proof}

By Lemma~\ref{lemma:maximalLogTerm}, 
a logarithmic $A$-hypergeometric series has the property that the
coefficients of the maximal logarithmic monomials are logarithm-free
$A$-hypergeometric series of the same parameter.  
Such a series is holomorphic in 
the parameters by Theorem~\ref{thm:logFreeSeriesHolomorphic}.
Thus, in order to prove Theorem~\ref{thm:logSeriesHolomorphic}, 
it is enough to express the coefficients of logarithmic monomials 
of smaller exponents, in terms of coefficients of logarithmic monomials
with larger exponents, in such a way that holomorphy is preserved. 
To achieve this goal, we perform a change of variables, which is aided by the following result.

\begin{lemma}
\label{lemma:goodCone}
If $K$ is a pointed, full dimensional rational polyhedral cone in
$\RR^m$ and  $L$ is a full rank lattice in $\RR^m$,  
then there exists a $\ZZ$-basis
$\{\lambda_1,\dots,\lambda_m\}$ of $L$ such that $K\cap L$ is contained in the set 
$\{ \nu_1\lambda_1+\cdots +\nu_m \lambda_m
  \mid \nu_1,\dots,\nu_m \in\NN \}$.
\end{lemma}

\begin{proof}
Since $K$ is full dimensional and pointed, so is its polar cone $K^*$. 
Let $L^*$ be the dual lattice of $L$, that is, 
$L^* = \{ u\in\RR^m \mid 
  u \cdot \lambda \in \ZZ \text{ for all } \lambda \in L\}$. 
The group generated by the monoid $K^*\cap L^*$ is $L^*$, so the Hilbert basis of $K^*\cap L^*$ must contain a set of generators of $L^*$, say $\kappa_1,\dots,\kappa_m$. 
This implies that $K$ is contained in the cone given by $\nabla = \{ z 
  \mid \kappa_j \cdot z \geq 0 \text{ for } j=1,\dots,m\}$. 
Let $\lambda_1,\dots,\lambda_m$ be the dual basis of 
$\kappa_1,\dots,\kappa_m$; this is the desired basis of $L$ because by construction, $\nabla$ is the set of nonnegative linear combinations of 
$\lambda_1,\dots,\lambda_m$. 
\end{proof}

We state one more result required for the proof of Theorem~\ref{thm:logSeriesHolomorphic}.

\begin{lemma}
\label{lemma:hadamardProduct}
Assume that 
\begin{equation}
\label{eqn:Hadamard}
\varphi_1(z,\alpha) = \sum_{\mu \in \NN^m} c(\mu, \alpha) \, z^\mu 
\quad\text{and}\quad 
\varphi_2(z,\alpha)=\sum_{\mu \in \NN^m} c'(\mu, \alpha) \, z^\mu
\end{equation}
are holomorphic for $(z,\alpha)$ in a product of domains $U_1\times U_2 \subset
\CC^m \times \CC^k$, where $U_1$ contains the origin in $\CC^m$. In
particular, for each fixed $\alpha \in U_2$, the series~\eqref{eqn:Hadamard}
are absolutely convergent on $U_1$.
Then, there is a domain $U_3 \subset \CC^m$ such that the Hadamard product 
\[
(\varphi_1*\varphi_2)(z,\alpha) \defeq 
\sum_{\mu \in \NN^m} 
  c(\mu,\alpha)\,c'(\mu,\alpha)\,z^\mu 
\]
is holomorphic on $U_3\times U_2$.
\end{lemma}

\begin{proof}
After possibly rescaling, we can assume that $U_1$ is the polydisc 
$\{z \, | \, |z_i| < 1, i = 1, \dots, n\}$. In this case, if $(z, \alpha) \in U_1\times U_2$, then so is 
$(\sqrt{|z|}, \alpha)$, where $\sqrt{|z|} = (\sqrt{|z_1|}, \dots, \sqrt{|z_n|})$.

Let $\alpha_0 \in U_2$. Then, there are absolutely convergent power series expansions
\[
\varphi_1(z,\alpha) = \sum_{\mu, \eta} c_{\mu, \eta}\,(\alpha-\alpha_0)^\eta\, z^\mu
\quad \text{and} \quad 
\varphi_2(z,\alpha) = \sum_{\mu, \eta'} c'_{\mu, \eta}\,(\alpha-\alpha_0)^{\eta'}\, z^\mu,
\]
where $c(\mu, \alpha) = \sum_\eta c_{\mu, \eta}(\alpha - \alpha_0)^\eta$, and ditto
for $c'(\mu, \alpha)$. We have that, formally,
\begin{align*}
(\varphi_1*\varphi_2)(z,\alpha) 
&= \sum_{\mu, \eta, \eta'} c_{\mu, \eta}\,c'_{\mu, \eta'}\,(\alpha-\alpha_0)^\eta (\alpha-\alpha_0)^{\eta'} z^\mu\\
& = \sum_{\mu, \delta} C_{\mu, \delta}\,(\alpha-\alpha_0)^{\delta}\,z^\mu,
\end{align*}
where $C_{\mu, \delta} = \sum_{\eta+\eta' = \delta} c_{\mu, \eta}\,c'_{\mu, \eta'}$.
Thus, the statement follows from the elementary estimates
\begin{align*}
\sum_{\mu, \delta} \left|C_{\mu, \delta}\right|\,|\alpha-\alpha_0|^{\delta}\,|z|^\mu 
& \leq \sum_{\mu, \eta, \eta'} |c_{\mu, \eta}|\,|c'_{\mu, \eta'}|\,|\alpha-\alpha_0|^\eta |\alpha-\alpha_0|^{\eta'} |z|^\mu\\
& \leq
\left(\sum_{\mu, \eta} |c_{\mu, \eta}|\,|\alpha-\alpha_0|^\eta\,|z|^{\frac\mu2}\right)
\left(\sum_{\mu, \eta'} |c'_{\mu, \eta'}|\, |\alpha-\alpha_0|^{\eta'} \,|z|^{\frac\mu2}\right).\qedhere
\end{align*}
\end{proof}

\begin{proof}[Proof of Theorem~\ref{thm:logSeriesHolomorphic}]
Consider a solution $\varphi(x,\bar \alpha)$ of $H_A(\bar{\beta})$ as in~\eqref{eqn:barvarphi}.
It is given that $C^*$ is pointed, and we may assume that $C^*$ is full dimensional. 
Note that $m=\rank(\ker_{\ZZ}(A)) = n-d$. Therefore, Lemma~\ref{lemma:goodCone} applied to 
the cone $C^* \cap \ker_{\RR}(A) \subset \ker_{\RR}(A)\cong\RR^m$ 
and the lattice $\ker_{\ZZ}(A)$
yields a basis $b_1,\dots,b_m$ of $\ker_{\ZZ}(A)$ such that 
\[
C^* \cap \ker_{\ZZ}(A) \subseteq 
  \{ \nu_1b_1+\cdots+\nu_mb_m \mid 
 \nu_1,\dots,\nu_m \in\NN\}.
\] 
Moreover, since $b_1,\dots,b_m$ is a lattice basis, any element
$u\in C^* \cap \ker_{\ZZ}(A)$ may be expressed as $u = B\nu$ for some $\nu \in \NN^m$, where $B$ is the matrix whose columns are given by $b_1,\dots,b_m$. 
Therefore, we may write 
\[
\varphi(x; \bar \alpha) 
  = x^{\bar{\alpha}}\sum_{\nu \in \NN^m} 
  p_{B \nu}(\log(x); \bar \alpha)\,x^{B\nu},
\]
where the $p_u$ are polynomials in $n$ variables whose coefficients
depend on $\bar\alpha$. The support of the series above may be strictly contained
in $\NN^m$, however, it is more convenient to use this larger summation set.
Furthermore, by~\cite[Proposition~5.2]{saito-logFree}, 
for each fix $\bar \alpha$, the polynomial $p_u$ 
(or, $p_{B\nu}$) 
belongs to the symmetric algebra of the lattice $\ker_\ZZ(A)$. 
Thus, there are $m$-variate polynomials
$q_\nu(y_1,\dots,y_m;\bar\alpha) \in \CC[y_1,\dots,y_m]$ (whose
coefficients depend on $\bar\alpha$)
such that 
\[
p_{B \nu} (\log(x); \bar \alpha) 
  = q_\nu(
  \log(x^{b_1}),\dots,\log(x^{b_m}); \bar \alpha).
\] 
Consequently, the series  
\[
F(z; \bar \alpha)
  \defeq \sum_{\nu \in \NN^m} 
  q_\nu(\log(z); \bar \alpha)\,z^{\nu}
\]
is such that
\[
\varphi(x; \bar \alpha)
  = x^{\bar{\alpha}}\,F(x^{b_1},\dots,x^{b_m}; \bar \alpha).
\] 
We refer to this passage from the $n$-variate series $\varphi$ to the $m$-variate series $F$ as \emph{dehomogenizing the torus action}. 
A $D$-module theoretic study of this procedure and its implications for hypergeometric systems can be found in~\cite{BMW13-holSing}. 
Note that, by construction of the convergence domain in $x$ of
$\varphi$, 
the series $F(z)$ converges for $z$ in a polydisc $W$ around the
origin in $\CC^m$.
Moreover, by Theorem~\ref{thm:openSetInx} the same neighborhood $W$ can be used for any parameter.

Using the notation from in Corollary~\ref{coro:logSeriesInParameters},
let $\alpha \in (\CC \minus \ZZ)^\sigma$. 
For convenience, we assume that $1 \in \sigma$, and
we fix the value of all parameters (the coordinates of $\alpha$) except $\alpha_1$.
Then for $\alpha_1 \in \CC \minus \ZZ$, $\varphi(x;\alpha)$ is a
well defined solution of $H_A(A\cdot \alpha)$; each coefficient
$p_u(\log(x);\alpha)$ is a polynomial in $\log(x)$ with coefficients
that are rational functions in $\alpha_1$. Therefore, each $q_\nu$ is
a polynomial in $\log(z)$ whose coefficients are rational functions of
$\alpha_1$. The degrees of the numerators and denominators of these
rational functions need not be uniformly bounded.

Since $x^{\alpha}$ is entire as a function of 
$\alpha_1$, the series $\varphi$ is holomorphic for 
$\alpha_1 \in \CC \minus \ZZ$ if and only if $F$ is holomorphic for
$\alpha_1 \in \CC \minus \ZZ$.

We can write
\begin{equation}
\label{eqn:logHornSeries}
F (z;\alpha)
  = \sum_{\gamma \in S} 
  \log(z)^\gamma F_\gamma(z;\alpha), 
\end{equation}
where $S\subset \NN^m$ is a finite set. 
Note that each series $F_\gamma(z;\alpha)$ is logarithm-free. 
The coordinatewise maximal elements of $S$ correspond to
coordinatewise maximal logarithmic monomials as in
Lemma~\ref{lemma:maximalLogTerm}.  
This implies that the series $F_\delta$ corresponding to
coordinatewise maximal elements $\delta$ of $S$ are dehomogenizations of
logarithm-free $A$-hypergeometric series. 
Note also, that these logarithm free $A$-hypergeometric series have
exponents that depend on $\alpha_1$, and are defined for $\alpha_1 \in
\CC \minus \ZZ$. Using the
expresion~\eqref{eqn:seriesFor-alpha} of a logarithm-free
$A$-hypergeometric series, and
Theorem~\ref{thm:logFreeSeriesHolomorphic}, 
we conclude that
these series are holomorphic for $\alpha_1 \in \CC \minus \ZZ$ and $z
\in U'$.

The above observation is the base case in an inductive argument to
show that the series $F_\gamma(z;\alpha)$ 
is holomorphic in $\alpha$. For the inductive step,
we wish to express $F_\gamma(z,\alpha)$ in terms series
$F_\delta(z,\alpha)$ for which $\delta$ is larger than $\gamma$
componentwise.

Let $B_1,\dots,B_n \in \ZZ^m$ denote the rows of the $n\times m$ matrix $B$
obtained in the process of dehomogenizing the torus action.
For $k=1,\dots, m$, recall the polynomials in the variables $\mu =
(\mu_1,\dots,\mu_m)$ from~\eqref{eqn:HornEquations}.
\[
P_k(\mu; \alpha) 
 = \prod_{b_{jk}>0} \prod_{\ell=0}^{b_{jk}-1} 
  (B_j \cdot \mu + \alpha_j - \ell)
\quad\text{and}\quad 
Q_k(\mu; \alpha) 
  = \prod_{b_{jk}<0} \prod_{\ell=0}^{|b_{jk}|-1}
  (B_j \cdot \mu + \alpha_j - \ell).
\]
We consider the \emph{Horn operators} associated to $B$, defined as
\begin{equation}
\label{eqn:HornOp}
H_k \defeq Q_k(z_1\del_{z_1},\dots,z_m\del_{z_m}; \alpha) - z_k
  P_k(z_1\del_{z_1},\dots,z_m\del_{z_m}; \alpha), \quad k=1, \dots, m.
\end{equation}
Given a vector $b\in\ZZ^n$, let $b_+, b_-\in\NN^n$ be the unique vectors such that $b=b_+-b_-$. 
Let $b_1,\dots,b_m$ denote the columns of $B$. For $k=1,\dots,m$, the operators 
$\del^{b_{k+}} - \del^{b_{k-}}$ annihilate 
$\varphi$. Using this fact, we can show that there exists $\omega \in \CC^m$,
depending linearly on $\alpha$, such
that $z^\omega F(z;\alpha)$ is a solution of the system defined by the
Horn operators $H_1,\dots,H_m$ (see~\cite{BMW13-holSing} for
details). In what follows, we assume that 
$\omega=0$ to simplify notation, and observe that all our arguments
are directly applicable to the general case.

Let $k$ be a fixed integer in between $0$ and $m$.
As $F(z;\alpha)$ is a solution to the $k$th Horn operator,
the coefficient of the logarithmic monomial $\log(z)^\gamma$ in the series obtained
by applying the $k$th operator~\eqref{eqn:HornOp}
to $F(z;\alpha)$ must  vanish.
This coefficient is the sum of the $k$th Horn operator applied to the
series  $F_\gamma$ and a series arising from applying differential 
operators to coefficients of $F_\delta$, for $\delta$ that are coordinatewise larger than $\gamma$. 
Thus,
\begin{equation}
\label{eqn:HornApplied}
H_k\bullet 
  F_\gamma(z;\alpha) = G_k(z;\alpha),
\end{equation}
where, by the inductive hypothesis, the series $G_k(z,\alpha)$ 
is holomorphic for $\alpha_1$ in $\CC\minus\ZZ$ and $z \in U_{k,\gamma}$
where $U_{k,\gamma}$ is a polydisc around $0$ which depends only on
$k$ and $\gamma$, and not on $\alpha$.

Let us expand the series $F_\gamma(z;\alpha)$  as
\[
F_\gamma(z;\alpha) 
  = \sum_{\mu \in \NN^m} f_\mu(\alpha) z^\mu.
\]
Our goal is to show that $F_\gamma$ is holomorphic in $\alpha$ by performing a second induction, 
over the dimension of the index set $\NN^m$. 
Let $\eta \in \NN^r$ for some $1\leq r \leq m$. We identify
$\eta$ with its image under the natural injection $\NN^r \rightarrow \NN^r \times \NN^{m-r}$.
Furthermore, let $\hat \eta$ be the image of $\eta$ under the projection $\NN^r \rightarrow \NN^{r-1}$
given by forgetting the last coordinate. 
We also use the notation $(\hat \eta, \ell) = (\eta_1, \dots, \eta_{r-1}, \ell)$. 
In our second induction, we consider the $r$-variate partial sums 
\begin{equation}
\label{eqn:Fr}
F_\gamma^r(z;\alpha) 
  = \sum_{\eta \in \NN^r} f_\eta(\alpha)\,z^\eta.
\end{equation}
As basis for the induction, we have that
that $f_0(\alpha)$ is a rational function of $\alpha_1$ without poles
in $\CC \minus \ZZ$, that is independent of $z$.

Expand the series $G_k(z;\alpha)$ as 
\[
G_k(z;\alpha) 
  = \sum_{\mu \in \NN^m} g^{(k)}_\mu(\alpha) 
  z^\mu.
\]
Identifying coefficients of \eqref{eqn:HornApplied}, we find that
\[
Q_k(\mu+e_k; \alpha) f_{\mu+e_k}(\alpha) - P_k(\mu; \alpha)f_\mu(\alpha) = g^{(k)}_{\mu+e_k}(\alpha).
\]

For each $k$, his is a first order inhomogeneous recurrence which can be solved explicitly,
we present here only the solution for $\eta \in \NN^r$ in the case that $k=r$;
\begin{equation}
\label{eqn:firstSeries}
f_{\eta}(\alpha) = \bigg[\prod_{\ell=0}^{\eta_r-1}
  \frac{P_r((\hat \eta,\ell);\alpha)}{Q_r((\hat\eta,\ell+1);\alpha)} \bigg] 
  \bigg[
  f_{\hat \eta}(\alpha)
  + \sum_{j=0}^{\eta_r-1} 
  \left(\frac{g^{(r)}_{(\hat \eta, j+1)}(\alpha)}{Q_r((\hat \eta, j+1);\alpha)}
  \Bigg/
  {\prod_{\ell=0}^{j}\frac{P_r((\hat \eta, \ell);\alpha)}{Q_r((\hat \eta,\ell+1);\alpha)}} \right)\bigg].
\end{equation}
Thus, the series with terms~\eqref{eqn:Fr} can be expanded as a sum of two series. We consider
each of them separately.

The first summand is the Hadamard product of the two series
\begin{equation}
\label{eqn:firstSmallSeries}
 \sum_{\eta\in \NN^r} 
  \prod_{\ell=0}^{\eta_r-1}
  \frac{P_r((\hat \eta,\ell);\alpha)}
  {Q_r((\hat \eta,\ell+1);\alpha)}
  \,z^\eta 
\end{equation}
and
\[
\frac{1}{1-z_r}\sum_{\hat\eta\in \NN^{r-1}} 
  f_{\hat \eta}(\alpha)\,z^{\hat \eta}
 = 
 \sum_{\eta\in \NN^{r}} 
  f_{\hat \eta}(\alpha)\,z^{\eta}
\]
The latter is holomorphic for $\alpha_1\in \CC \minus \ZZ$ and $z$ in
a polydisc around $0$ that does not depend on $\alpha$
by the second induction
hypothesis.
To see that~\eqref{eqn:firstSmallSeries} is holomorphic for $\alpha_1
\in \CC \minus \ZZ$ and $z$ in a polydisc around $0$ that does not
depend on $\alpha$, 
we bound this series term by term in absolute value using the series $\Phi_B$
from~\eqref{eqn:hornSeries}, and apply
Theorem~\ref{thm:HornSeriesConverge} and its proof.
We conclude, by Lemma~\ref{lemma:hadamardProduct}, the first summand
of \eqref{eqn:firstSeries} is holomorphic for
$\alpha_1 \in \CC \minus \ZZ$ and $z$ in a polydisc around $0$ that
does not depend on $\alpha$.

The series whose terms are given by the second summand in~\eqref{eqn:firstSeries}
can be rewritten as
\begin{equation}
\label{eqn:secondSum}
\sum_{\eta \in \NN^r} \bigg[\prod_{\ell=0}^{\eta_r-1}
   \frac{P_r((\hat\eta,\ell),\alpha)}
  {Q_r((\hat \eta,\ell+1), \alpha)} 
  \bigg]
  \bigg[ \sum_{j=0}^{\eta_r-1} 
  \frac{g^{(r)}_{(\hat\eta, j+1)}(\alpha)}{Q_r((\hat\eta, j+1), \alpha)}
  \prod_{\ell=0}^{j}\frac{Q_r((\hat \eta,\ell+1), \alpha)}{P_r((\hat\eta,\ell), \alpha)} 
  \bigg] z^\eta.
\end{equation}
The series
\[
\sum_{\eta \in \NN^r} 
  \frac{Q_r((\hat \eta,\ell+1), \alpha)}{P_r((\hat \eta,\ell), \alpha)}\,  z^\eta
\qquad\text{and}\qquad
\sum_{\eta \in \NN^r}
  \frac{z^\eta }{Q_r(\eta+e_r, \alpha)}
\]
define holomorphic functions for $\alpha_1 \in \CC \minus \ZZ$ and $z$
in a polydisc around $0$ independent of $\alpha$, again using~Theorem~\ref{thm:HornSeriesConverge}.
The (first) inductive hypothesis gives that the series
$\sum_{\eta\in\NN^r} g^{(r)}_\eta(\alpha) z^\eta$ is holomorphic for
$\alpha_1 \in \CC \minus \ZZ$ and $z$ in a
polydisc around $0$ that does not depend on $\alpha$
Taking Hada\-mard products, Lemma~\ref{lemma:hadamardProduct} implies that
the series
\begin{equation}\label{eqn:gQQPSum}
\sum_{\eta\in\NN^r} 
  \left[ 
  \frac{g^{(r)}_{\eta+e_r}(\alpha)}{Q_r(\eta+e_r, \alpha)}
  \prod_{\ell=0}^{\eta_r} 
  \frac{Q_r((\hat\eta,\ell+1), \alpha)}{P_r((\hat\eta,\ell), \alpha)} 
 \right] z^\eta
\end{equation}
is holomorphic for $\alpha_1 \in \CC \minus \ZZ$ and $z$ in a polydisc
around $0$ that does not depend on $\alpha$.
Taking the product of~\eqref{eqn:gQQPSum} with the function 
$\sum_{\eta\in \NN^r} z^\eta$ 
we see that the series whose coefficients consist of the second factor of \eqref{eqn:secondSum}
defines a holomorphic function for $\alpha_1 \in \CC \minus \ZZ$ and
$z$ in a polydisc around $0$ that does not depend on $\alpha$.
Finally,~\eqref{eqn:secondSum} is the Hadamard product
of~\eqref{eqn:firstSmallSeries} and~\eqref{eqn:gQQPSum}, so by
Lemma~\ref{lemma:hadamardProduct},~\eqref{eqn:secondSum} is
holomorphic for $\alpha_1 \in \CC \minus \ZZ$ and $z$ in a polydisc
around $0$ that does not depend on $\alpha$.

Note that in the above process, the convergence domain in $z$ of the
series involved has to be shrunk several times as we build $F_\gamma$
inductively. At each step, the domains in $z$ that are used are
independent of the parameters $\alpha$. As a matter of fact, what
these domains are depends only on the initial data of the matrix $A$,
the weight vector $w$ (when we use Theorem~\ref{thm:openSetInx}), the
matrix $B$ (when we use Theorem~\ref{thm:HornSeriesConverge}), and the
indices of the inductions (which control the Hadamard products that
need to be performed). Moreover, both inductions are finite: the
second induction has $m$ steps, while the first has at most
$(2^{2d}\vol(A))^n$ steps. We conclude that there exists 
a polydisc $W'$ around $0$ in $\CC^m$ such that $F(z;\alpha)$ is
holomorphic for $(z,\alpha_1) \in W' \times (\CC \minus \ZZ)$.
\end{proof}

\begin{proof}[Proof of Theorem~\ref{thm:mainHomogeneous}]
The stratification $\mathscr{S}$ introduced in Section~\ref{sec:intro}
has the following properties. If $\beta \in \mathscr{S}_i$ then
$\beta$ belongs to a unique irreducible component $S$ of the closure
$\overline{\mathscr{S}}_i$, and this component is a codimension $i$
affine subspace of $\CC^d$. If in addition, $\beta$ belongs to an
integer translate $L$ of the span of a face $\sigma$ of the triangulation
$\Delta_w$, then $L \supseteq S$.

Consider the exponents of $H_A(\bar\beta)$ with respect to $w$, where
$\bar\beta \in \mathscr{S_i}$. If $\bar\alpha$ is such an exponent, then
$\sigma = \{ i \mid \bar\alpha_i \notin \ZZ\}$ is a simplex in
$\Delta_w$. By the previous argument, the closure of
$\{ A \cdot \alpha \mid \alpha \in \CC^n; \alpha_i =
\bar\alpha_i,\; i \notin \sigma; \alpha_j \notin \ZZ, \; j \in \sigma\}$
contains the irreducible component of $\bar\beta$ in 
$\overline{\mathscr{S}}_i$. Thus, by Theorem~\ref{thm:logSeriesHolomorphic} the
hypergeometric series corresponding to the exponent $\bar\alpha$ can be extended to
a holomorphic function for $(x,\alpha)$ in $U\times Y$, where $Y$ is a
neighborhood of $\bar\beta$ in $\mathscr{S}_i$.
\end{proof}

\section{The confluent case}
\label{sec:confluentSolutions}

In this section, we use ideas from~\cite{berkesch,nilsson} to 
prove Theorem~\ref{thm:mainInhomogeneous}. 
We assume in this section that $A$ is not necessarily homogeneous. 
Let 
\[ 
{\rho(A)} = 
\left[ \begin{array}{c|c}
1 & 1\  \cdots \ 1 \\ \hline
0 & \\
\vdots & A \\
0 & 
\end{array} \right]
\] 
be the \emph{homogenization} of $A$. 
Observe that $\ZZ(\rho(A)) = \ZZ^{d+1}$ and ${\rho(A)}$ is homogeneous.

\begin{theorem}\label{thm:homogRank+isomSol}
If $\beta\in\CC^d$ and $\beta_0\in\CC$ generic, 
then the modules corresponding to the $A$-hyper\-geo\-metric systems
$\HA$ and $H_{\rho(A)}(\beta_0,\beta)$ have the same rank. 
In addition, for $x\in\CC^{n+1}$ sufficiently generic, the following map is an isomorphism: 
\[
\Sol_x(H_{\rho(A)}(\beta_0,\beta))\to\Sol_{\overline{x}}(\HA),
\] 
where $\overline{x}$ is obtained from $x$ by setting $x_0$ equal to $1$.
\end{theorem}

This result can be used to produce an upper bound for the rank of $\HA$ in the confluent case. 

\begin{corollary}\label{cor:upperBound}
If $A\in\ZZ^{d\times n}$ and $\beta\in\CC^d$, then the rank of $\HA$ is bounded above by $2^{2d+2}\cdot \vol(A)$. 
\end{corollary}
\begin{proof}
\cite[Corollary~4.1.2]{SST} states that if $A$ is homogeneous, then 
\[
\rank (D/\HA)\leq 2^{2d}\cdot\vol(A) 
\quad\text{for any $\beta$}.
\] 
The result is then follows directly from Theorem~\ref{thm:homogRank+isomSol}, since $\vol(A) = \vol({\rho(A)})$. 
\end{proof}

To prove Theorem~\ref{thm:homogRank+isomSol}, we first recall a result from~\cite{berkesch} that computes the rank of $\HA$ for any $\beta$. 
For a face $F\preceq A$, 
let $\vol(F)$ denote the normalized volume of $F$ in $\ZZ F$, 
and consider the translates 
\[
\EE_F^\beta\defeq 
\big[\ZZ^d\cap(\beta+\CC F) \big]\minus(\NN A+\ZZ F) = \bigsqcup_{b\in R_F^\beta} (b+\ZZ F),
\]
where $R_F^\beta$ is a set of lattice translate representatives. 
Each lattice translate in an $\EE_F^\beta$ is called a \emph{ranking lattice} of $\beta$. Let  
\begin{align}
\label{eqn:rankingLattices}
\EE^\beta = \bigcup_{F\preceq A}\EE_F^\beta 
= \bigcup_{\stackrel{F\preceq A}{b\in R_F^\beta}} (b+\ZZ F)
\end{align}
denote the union of ranking lattices.

\begin{theorem}
\label{thm:rankComputation}
For any $A\in\ZZ^{d\times n}$ and $\beta\in\CC^d$, 
the rank 
of $\HA$ 
can be computed from the combinatorics of the ranking lattices in $\EE^\beta$. 
The formula obtained uses only the lattice of intersections of the faces with maximal lattice translates appearing in $\EE^\beta$ and, for each such face $F$, $\vol(F)$, $\codim(F)$, and $|R_F^\beta|$. 
\qed
\end{theorem}

In light of Theorem~\ref{thm:rankComputation}, 
the proof of Theorem~\ref{thm:homogRank+isomSol} will proceed by showing that there is a bijection between the ranking lattices in $\EE^\beta$ and the ranking lattices in $\EE^{(\beta_0,\beta)}$, and, in addition, the combinatorial data in the formula in Theorem~\ref{thm:rankComputation} are preserved by homogenization. 

We define the \emph{homogenization map} 
$\rho\colon\CC^d~\rightarrow~\CC^{d+1}$ to be 
\[
\rho(b) = \left[ \begin{smallmatrix} 1 \\ b \end{smallmatrix} \right]
\]
for $b\in\CC^d$. 
Note that the homogenization of $A$ is 
${\rho(A)} = [ \rho(a_0)\ \ \rho(a_1) \ \cdots \ \rho(a_n) ]$, 
where $a_0 = 0$ in $\CC^d$. 
Given a face $F\preceq A$, setting 
\[
\rho(F) = \{ \rho(0) \} \cup \{ \rho(a_i) \mid a_i\in F\} \subseteq {\rho(A)} 
\] 
induces a one-to-one correspondence between the faces of $A$ 
and the faces of ${\rho(A)}$ containing $\rho(0)$.

\begin{proposition}
\label{prop:preserve} 
For a face $F\preceq A$, the homogenization map $\rho$ preserves 
its codimension, volume, and lattice index $[\ZZ A \cap \RR F: \ZZ F]$. 
\end{proposition}
\begin{proof}
Homogenization preserves codimension because $\dim(\rho(F)) = \dim(F) +1$.
For volume, notice that the convex hull of $\rho(F)$ 
and the origin in $\ZZ(\rho(F)) \otimes_\ZZ \RR$ is a cone over 
the convex hull of $F$ and the origin in $\ZZ F \otimes_\ZZ \RR$, 
under the obvious embedding, 
so this result follows from the definition of volume.
Finally, if $b,c\in \ZZ^d$, then 
$b-c\in\ZZ F$ if and only if 
$\rho(b)-\rho(c)\in\ZZ(\rho(F))$, 
since $\rho(0)\in \rho(F)$. 
\end{proof}

\begin{proposition}\label{prop:sameLatticeTranslates}
If $\beta\in\CC^d$ and $\beta_0\in\CC$ is generic, 
then there is a bijection between the ranking lattices in $\EE^\beta$ and $\EE^{(\beta_0,\beta)}$.
In particular, if 
\[
\cJ(\beta) = 
\left\{ (F,b) \ \Big\vert\ F\preceq A,\ b\in\ZZ^d,\ \text{and } (b+\ZZ F)\subseteq\EE_F^{\beta} \right\},
\]
then 
\[
\EE^{(\beta_0,\beta)} = \bigcup_{(F,b)\in \cJ(\beta)} \left[ \rho(b) + \ZZ(\rho(F)) \right]. 
\] 
\end{proposition}
\begin{proof}
Fix $F\preceq A$ and $b\in\ZZ^d$. 
Since $\rho(0) \in \rho(F)$, 
$(b + \ZZ F)\cap(\NN A + \ZZ F) = \varnothing$ if and only if 
\[
[\rho(b) + \ZZ(\rho(F))]\cap[\NN(\rho(A)) + \ZZ(\rho(F))] = \varnothing.
\]
Again because $\rho(0) \in \rho(F)$, 
$b\in\beta+\CC F$ is equivalent to 
$\rho(b)\in \rho(\beta) +\CC(\rho(F)) = \beta' + \CC(\rho(F))$. 
Therefore, if $b+\ZZ F\subseteq \FF^\beta$, then 
$\rho(b)+\ZZ\rho(F)\subseteq\EE^{(\beta_0,\beta)}$. 
Further, for the reverse containment, it is enough to show that a face $G\preceq \rho(A)$ with $\rho(0)\notin G$ does not contribute additional ranking lattices to $\EE^{(\beta_0,\beta)}$. 

Since $\{\rho(0)\mid 0\in\ZZ^d\}$ is a face of $\rho(A)$,  
the union of ranking lattices 
$\bigcup_{\beta_0\in\RR_{\geq 0}} \EE^{(\beta_0,\beta)}$
necessarily involves finitely many lattice translates. 
Thus there is a Zariski open set from with to choose $\beta_0\in\CC$ so that any additional lattice translates 
from a $G\preceq \rho(A)$ with $\rho(0)\notin G$ can be avoided. 
In particular, for all $\beta\notin\RR_{\geq 0}A$, only finitely many lattices which intersect $\RR_{\geq 0}A$ need be avoided, and then any $\beta_0\in\RR_{\gg 0}$ will yield a proper choice.  
\end{proof}

\begin{proof}[Proof of Theorem~\ref{thm:homogRank+isomSol}]
The first statement follows immediately from Theorem~\ref{thm:rankComputation} and Proposition~\ref{prop:preserve}. 
Thus it is enough to show that the map 
\[
\Sol_x(H_{\rho(A)}(\beta_0,\beta))\to\Sol_x(\HA)
\]
given by setting $x_0$ equal to $1$ is injective. 
But this is the natural map induced by 
\[
\CC[\del_1,\dots,\del_n]/I_A\hookrightarrow \CC[\del_0,\del_1,\dots,\del_n]/I_{\rho(A)}
\] 
by \cite[Proposition~3.17]{nilsson}, thanks to the proof of~\cite[Lemma~3.12]{nilsson}. Since the resulting $D$-modules have the same rank, the proof is complete.
\end{proof}

\begin{proof}[Proof of Theorem~\ref{thm:mainInhomogeneous}]
In light of the proof of Theorem~\ref{thm:mainHomogeneous} at the end
of Section~\ref{sec:LogSeriesHolomorphic}, it remains only to reduce
the result when $A$ is not homogeneous to the homogeneous case. We
apply Theorem~\ref{thm:mainHomogeneous} to $\rho(A)$ using a weight
vector which is a generic perturbation of $(0,1,1,\dots,1)\in
\RR^{n+1}$, and produce an open set $U' \subset \CC^{n+1}$ and a
stratification $\mathscr{S}'$ such that the solutions of
$H_{\rho(A)}(\beta_0,\beta)$ are holomorphic on $U'\times S'$ for any
stratum $S'$ induced by $\mathscr{S}'$. By construction, the
stratification $\mathscr{S}$ we are interested in is the projection of
$\mathscr{S}'$ that forgets the zeroth coordinate.
By the end of the proof of Proposition~\ref{prop:sameLatticeTranslates}, 
it is possible to find real numbers 
$\alpha_0,\alpha_1,\dots,\alpha_n$ that are algebraically independent and irrational so that 
$Z = \Var(\alpha_0x_0+\alpha_1x_1+\cdots+\alpha_nx_n)\subset\CC^{d+1}$ satisfies 
$\pi(Z) = \CC^d$ and 
$Z\cap [(\ZZ^{d+1}+\ZZ G)\minus\NN A] = \varnothing$ for all $G\preceq\rho(A)$ with $\rho(0)\notin G$. 
As such, for each $\beta\in\CC^d$, $Z$ contains a unique lift
$(\beta_0,\beta)$ of $\beta$ so that
Theorem~\ref{thm:homogRank+isomSol} holds, and the isomorphism therein
is holomorphic on $U\times S$ for any stratum $S$ induced by
$\mathscr{S}$, where $U$ is the projection of $U'$ that forgets the
zeroth coordinate.
\end{proof}

\raggedbottom
\def\cprime{$'$} \def\cprime{$'$}
\providecommand{\MR}{\relax\ifhmode\unskip\space\fi MR }
\providecommand{\MRhref}[2]{%
  \href{http://www.ams.org/mathscinet-getitem?mr=#1}{#2}
}
\providecommand{\href}[2]{#2}
\end{document}